\documentclass[12pt]{article}
\usepackage{amsmath,amssymb,amsthm,latexsym}
\usepackage{graphicx,xcolor,tikz,lineno}
\renewenvironment{proof}{\noindent{\bf Proof.}}{~~$\Box$}
\theoremstyle{plain}
\newtheorem{thrm}{Theorem}[section]
\newtheorem{lemm}[thrm]{Lemma}
\newtheorem{prot}[thrm]{Proposition}

\newtheorem{corl}[thrm]{Corollary}

\theoremstyle{definition}
\newtheorem{dfnt}[thrm]{Definition}
\newtheorem{remk}[thrm]{Remark}
\newtheorem{exa}[thrm]{Example}

\newcommand{\MySs}{
\begin{tikzpicture}[baseline]
\draw (0,0em) -- (0em,0.6em) -- (0.6em,0.6em) -- (0.6em,0em) -- (0,0em);
\draw (0,0em) parabola (0.6em, 0.6em);
\end{tikzpicture}}
\begin{document}
 \title{\bf On the pointwise supremum of the set of copulas with a given curvilinear section}
\author{
Yao Ouyang$\sp{a}$\thanks{oyy@zjhu.edu.cn}\qquad Yonghui Sun$\sp{a}$\thanks{syh\_0729@163.com}\qquad Hua-peng Zhang$\sp{b}$\thanks{Corresponding author. huapengzhang@163.com}\\
 $\sp{a}$\small\it Faculty of Science, Huzhou University, Huzhou, Zhejiang 313000, China \\
$\sp{b}$\small\it School of Science, Nanjing University of Posts and Telecommunications, \\ \small\it Nanjing 210023, China \\
}
\date{}
\maketitle
\begin{abstract}
Making use of the total variation of particular functions, we give an explicit formula for the pointwise supremum of the set of all copulas with a given curvilinear section. When the pointwise supremum is a copula is characterized. We also characterize the coincidence of the pointwise supremum and the greatest quasi-copula with the same curvilinear section.

{\it Keywords:} Copula, Quasi-copula, Curvilinear section, Total variation
\end{abstract}
\section{Introduction}

In the realm of statistical modeling and data analysis, understanding and modeling the dependencies among variables is crucial for accurate predictions and informed decision-making. Traditional methods often rely on assumptions about the joint distribution of variables, which can be limiting and restrictive, especially when dealing with complex real-world data. Copula-based method \cite{GR10,J14,TWSH19} allows for the separate modeling of marginal distributions and the dependency structure among variables, providing a more nuanced and realistic representation of the underlying data-generating process. Thus it offers a flexible and versatile alternative to those traditional methods.

Copulas, which were first introduced by Sklar in~\cite{S59}, find uses in various domains such as finance and insurance. In finance, copulas have been instrumental in modeling the joint behavior of asset returns, risk management, and portfolio optimization~\cite{CLV04}. In insurance, they have been used to assess the risk of simultaneous occurrences of extreme events~\cite{FM13}.

Over the years, the investigation of (quasi-)copulas with a given diagonal section has attracted much attention~\cite{DJ08,NQRU04,NQRU08}, since the diagonal section $\delta_C$ provides some information about the tail dependence of a bivariate random vector and $\delta_C$ itself is the distribution function of the random variable $\max\{X, Y\}$ provided that the copula of $(X, Y)$ is $C$. For a given diagonal section, the smallest copula always exists, while the greatest one may not. Recently,  Kokol Bukov\v{s}ek et al.~\cite{KMS24} presented an explicit representation of the pointwise supremum of the set of all copulas with a given diagonal section.

The study of (quasi-)copulas with a given diagonal section has been extended to that with a given curvilinear section~\cite{BMJ19}.
For a given curvilinear section $\Gamma_\phi$, the set of all copulas (resp. quasi-copulas), denoted as $\mathbb{C}_{\Gamma_\phi}$ (resp. $\mathbb{Q}_{\Gamma_\phi}$), is nonempty. The smallest element of $\mathbb{C}_{\Gamma_\phi}$ and that of $\mathbb{Q}_{\Gamma_\phi}$ exist and coincide. Although the greatest element of $\mathbb{Q}_{\Gamma_\phi}$ exists, denoted as $A_{\Gamma_\phi}$, the greatest element of $\mathbb{C}_{\Gamma_\phi}$ may not.
The present authors~\cite{OSZ25} characterized when $A_{\Gamma_\phi}$ is a copula. However, the pointwise supremum of $\mathbb{C}_{\Gamma_\phi}$ is unclear compared with the diagonal case.

The central task of this paper is to explore an explicit formula for the pointwise supremum of $\mathbb{C}_{\Gamma_\phi}$.
After recalling some necessary knowledge that will be used throughout the paper in Section~\ref{Sec-Preliminaries}, we give an explicit formula for the pointwise supremum of $\mathbb{C}_{\Gamma_\phi}$ in Section~\ref{Sec-Formula}. When the pointwise supremum of $\mathbb{C}_{\Gamma_\phi}$ is a copula is characterized in Section~\ref{Sec-ChaCop}, while Section~\ref{Sec-ChaCoin} is devoted to characterizing the coincidence of the pointwise supremum of $\mathbb{C}_{\Gamma_\phi}$ and $A_{\Gamma_\phi}$.

\section{Preliminaries}\label{Sec-Preliminaries}

In this section, we recall some basic notions and results concerning (quasi-)copulas with a given curvilinear section.
\begin{dfnt}\cite{N06}
A (bivariate) \emph{copula} is a function $C: [0, 1]^2\to [0, 1]$ satisfying:\\
(C1) the boundary condition, \emph{i.e.},
 \[
 C(x, 0)=C(0, x)=0\ {\rm and} \ C(x, 1)=C(1, x)=x\ (\forall\,x\in [0, 1]);
 \]
(C2) the $2$-increasing property, \emph{i.e.},
$V_C(R)\geq 0$ for any rectangle $R=[x_1, x_2]\times [y_1, y_2]$,
where $V_C(R)=C(x_2, y_2)-C(x_2, y_1)-C(x_1, y_2)+C(x_1, y_1)$. We will write $V(\cdot)$ instead of $V_C(\cdot)$ if there is no confusion.
\end{dfnt}

\begin{dfnt}\cite{N06}
A (bivariate) \emph{quasi-copula} is a function $Q: [0, 1]^2\to [0, 1]$ satisfying (C1) and the following conditions:\\
(Q1) $Q$ is increasing in each variable;\\
(Q2) $Q$ is $1$-Lipschitz, \emph{i.e.}, for all $x_1, x_2, y_1, y_2\in [0, 1]$, it holds that
\[
|Q(x_1, y_1)-Q(x_2, y_2)|\leq |x_1-x_2|+|y_1-y_2|.
\]
\end{dfnt}
Note that (C1) and (C2) together imply both (Q1) and (Q2), \emph{i.e.}, every copula is a quasi-copula, but not vice versa.
Three important copulas are $W, M$ and $\Pi$ given by
\[W(x, y)=\max\{x+y-1, 0\}, M(x, y)=\min\{x, y\} \ {\rm and} \ \Pi(x, y)=xy.\]
$W$ and $M$ are known as the \emph{Fr\'{e}chet-Hoeffding bounds} for copulas and quasi-copulas since $W\leq Q\leq M$
for any quasi-copula $Q$, while $\Pi$ is known as the independence copula.

\begin{dfnt}\cite{NQRU04}
A \emph{diagonal function} is a function $\delta: [0, 1]\to [0, 1]$ with the properties:\\
(D1) $\delta(1)=1$ and $\delta(t)\leq t$ for all $t\in [0, 1]$;\\
(D2) $0\leq \delta(t_2)-\delta(t_1)\leq 2(t_2-t_1)$ for all $t_1, t_2\in [0, 1]$ with $t_1\leq t_2$.
\end{dfnt}

The diagonal section $\delta_Q: [0, 1]\to [0, 1]$ of any quasi-copula $Q$ defined by $\delta_Q(t)=Q(t, t)$ is a diagonal function. Conversely, for any diagonal function $\delta$, there exist quasi-copulas with diagonal section $\delta$,
and the smallest quasi-copula $B_\delta$ with diagonal section $\delta$ is given by
\[
B_\delta(x, y) =\left \{
        \begin {array}{ll}
         y-\min\limits_{t\in[y, x]}\hat{\delta}(t)
                  &\quad \text{if\ \ $y\leq x$}  \\[2mm]
           x-\min\limits_{t\in[x, y]}\hat{\delta}(t)
                  &\quad \text{if\ \ $y>x$,}
        \end {array}
       \right.
\]
while the greatest quasi-copula $A_\delta$ with diagonal section $\delta$ is given by
\[
A_\delta(x, y) =\left \{
        \begin {array}{ll}
         \min\left\{y, x-\max\limits_{t\in[y, x]}\hat{\delta}(t)\right\}
                  &\quad \text{if\ \ $y\leq x$}  \\[2mm]
           \min\left\{x, y-\max\limits_{t\in[x, y]}\hat{\delta}(t)\right\}
                  &\quad \text{if\ \ $y>x$,}
        \end {array}
       \right.
\]
where $\hat{\delta}: [0, 1]\to [0, 1]$ is defined by $\hat{\delta}(t)=t-\delta(t)$.
Interestingly, $B_\delta$ is also a copula (called the \emph{Bertino copula}), but $A_\delta$ may not be.
There have been characterization theorems for $A_\delta$ to be a copula (see~\cite{NQRU04, NQRU08}).
Denote by $\mathbb{C}_\delta$ the set of all copulas with given diagonal section $\delta$. The explicit formula of the pointwise supremum $\bar{C}_\delta$ of $\mathbb{C}_\delta$ has recently been given by Kokol Bukov\v{s}ek et al.~\cite{KMS24}:
\[
\bar{C}_\delta(x, y)=\left \{
        \begin {array}{ll}
         \min\left\{y, x-\dfrac{1}{2}\Big(\hat{\delta}(x)+\hat{\delta}(y)+\mathbb{V}_{y}^{x}(\hat{\delta})\Big)\right\}
                  &\quad \text{if\ \ $y\leq x$}  \\[2mm]
           \min\left\{x, y-\dfrac{1}{2}\Big(\hat{\delta}(x)+\hat{\delta}(y)+\mathbb{V}_{x}^{y}(\hat{\delta})\Big)\right\}
                  &\quad \text{if\ \ $y>x$,}
        \end {array}
       \right.
\]
where $\mathbb{V}_{x}^{y}(\hat{\delta})$ is the total variation of $\hat{\delta}$ on $[x, y]$.

A finite sequence $\{x_i\}_{i=0}^n\subseteq[a, b]$ is called a~\emph{partition} of an interval $[a, b]$ if it is increasing with $x_0=a$ and $x_n=b$. Denote by $\mathcal P[a, b]$ the set of all partitions of $[a, b]$. The~\emph{total variation} $\mathbb{V}_a^b(f)$ of a real-valued function $f$ on $[a, b]$ is defined by
\[
\mathbb{V}_a^b(f)=\sup\left\{\sum_{k=1}^{n}|f(x_k)-f(x_{k-1})|\mid \{x_i\}_{i=0}^n\in\mathcal P[a, b]\right\}.
\]
 If $\mathbb{V}_a^b(f)<+\infty$, then $f$ is differentiable almost everywhere on $[a, b]$ and $\mathbb{V}_a^b(f)\geq\int_a^b|f^\prime(t)|{\rm d}t$, where the equality holds if and only if $f$ is absolutely continuous~\cite{RF10}.
We use the convention that $\mathbb{V}_b^a(f)=-\mathbb{V}_a^b(f)$ and $\mathbb{V}_a^a(f)=0$. Clearly, $\mathbb{V}_a^b(f)=\mathbb{V}_a^c(f)+\mathbb{V}_c^b(f)$ for any $c\in [a, b]$, $\mathbb{V}_a^b(f)=|f(b)-f(a)|$ if $f$ is monotone,
and $\mathbb{V}_a^b(f+g)\leq \mathbb{V}_a^b(f)+\mathbb{V}_a^b(g)$ for any two real-valued functions $f, g$ on $[a, b]$.

\begin{remk}\label{remk-variation}
For any $\{x_i\}_{i=0}^n\in\mathcal P[a, b]$, it is easy to prove by induction that there exists $\{y_i\}_{i=0}^{2m}\in\mathcal P[a, b]$ such that the sets $\{x_i\}_{i=0}^n$ and $\{y_i\}_{i=0}^{2m}$ are equal, and $f(y_{2k-1})\geq\max\{f(y_{2k-2}), f(y_{2k})\}$ for all $k\in\{1, 2, \cdots, m\}$. Hence,
\begin{eqnarray*}
  \sum_{k=1}^{n}|f(x_k)-f(x_{k-1})|
 &=& \sum_{k=1}^{2m}|f(y_k)-f(y_{k-1})| \\
 &=& \sum_{k=1}^{m}(f(y_{2k-1})-f(y_{2k-2}))+\sum_{k=1}^{m}(f(y_{2k-1})-f(y_{2k})) \\
 &=& 2\sum_{k=1}^{m}f(y_{2k-1})-2\sum_{k=1}^{m-1}f(y_{2k})-f(a)-f(b).
\end{eqnarray*}
Therefore, it holds that
\[
\mathbb{V}_a^b(f)=\sup\left\{2\sum_{k=1}^{m}f(x_{2k-1})-2\sum_{k=1}^{m-1}f(x_{2k})-f(a)-f(b)\mid \{x_i\}_{i=0}^{2m}\in\mathcal P[a, b]\right\}.
\]
\end{remk}

The study of copulas with a given diagonal section has been extended to that with a given curvilinear section~\cite{BMJ19}.
Let $\Phi$ be the set of all functions $\phi: [0, 1]\to [0, 1]$ that are continuous and strictly increasing with $\phi(0)=0$ and $\phi(1)=1$.
Obviously, ${\rm id}\in\Phi$, where ${\rm id}$ is the identity map on $[0, 1]$, \emph{i.e.}, ${\rm id}(x)=x$ for all $x\in [0, 1]$.
For any $\phi\in\Phi$, the curvilinear section $g_{Q, \phi}: [0, 1]\to [0, 1]$ of a quasi-copula $Q$ defined by $g_{Q, \phi}(t)=Q(t, \phi(t))$ has properties:
\begin{itemize}\setlength{\itemindent}{10pt}
\item[{\rm (i)}] $\max\{0, t+\phi(t)-1\}\leq g_{Q,\phi}(t)\leq\min\{t, \phi(t)\}$ for all $t\in [0, 1]$;
\item[{\rm (ii)}] $0\leq g_{Q,\phi}(t_2)-g_{Q,\phi}(t_1)\leq t_2-t_1+\phi(t_2)-\phi(t_1)$ for all $t_1, t_2\in [0, 1]$ with $t_1<t_2$.
\end{itemize}
Note that property (ii) can be rewritten as: $g_{Q, \phi}(t)$ is increasing and $g_{Q, \phi}(t)-t-\phi(t)$ is decreasing.
For any $\phi\in\Phi$, denote by $\Delta_\phi$ the set of all functions $\Gamma_\phi: [0, 1]\to [0, 1]$ with properties (i) and (ii).
Particularly, if $\phi={\rm id}$, then properties (i) and (ii) are equivalent to (D1) and (D2), and thus $\Delta_{\rm id}$ is nothing else but the set of all diagonal functions.

For any $\phi\in\Phi$ and $\Gamma_\phi\in\Delta_\phi$, similarly to the diagonal case, there exist quasi-copulas with curvilinear section $\Gamma_\phi$,
and the smallest quasi-copula $B_{\Gamma_\phi}$ with curvilinear section $\Gamma_\phi$ is given by
\[
B_{\Gamma_\phi}(x, y) =\left \{
        \begin {array}{ll}
         y-\min\limits_{t\in[\phi^{-1}(y), x]}\tilde{\Gamma}_\phi(t)
                  &\quad \text{if\ \ $y\leq\phi(x)$}  \\[2mm]
           x-\min\limits_{t\in[x, \phi^{-1}(y)]}\hat{\Gamma}_\phi(t)
                  &\quad \text{if\ \ $y>\phi(x)$,}
        \end {array}
       \right.
\]
where $\tilde{\Gamma}_\phi, \hat{\Gamma}_\phi: [0, 1]\to [0, 1]$ are defined by $\tilde{\Gamma}_\phi(t)=\phi(t)-\Gamma_\phi(t)$ and $\hat{\Gamma}_\phi(t)=t-\Gamma_\phi(t)$. $B_{\Gamma_\phi}$ is also a copula, called the \emph{Bertino copula with curvilinear section} $\Gamma_\phi$~\cite{BMJ19}.
The greatest quasi-copula $A_{\Gamma_\phi}$ with curvilinear section $\Gamma_\phi$ is given by
\[
A_{\Gamma_\phi}(x, y) =\left \{
        \begin {array}{ll}
         \min\left\{y, x-\max\limits_{t\in[\phi^{-1}(y), x]}\hat{\Gamma}_\phi(t)\right\}
                  & \text{if\ \ $y\leq\phi(x)$}  \\[2mm]
           \min\left\{x, y-\max\limits_{t\in[x, \phi^{-1}(y)]}\tilde{\Gamma}_\phi(t)\right\}
                  & \text{if\ \ $y>\phi(x)$.}
        \end {array}
       \right.
       \]
However, $A_{\Gamma_\phi}$ may not be a copula. Quite recently, we~\cite{OSZ25} completely characterized when $A_{\Gamma_\phi}$ is a copula.

\section{Representation of the pointwise supremum of $\mathbb{C}_{\Gamma_\phi}$}\label{Sec-Formula}

Denote by $\mathbb{C}_{\Gamma_\phi}$ the set of all copulas with given curvilinear section $\Gamma_\phi$.
In this section, inspired by the work~\cite{KMS24} on the diagonal case, we first construct two copulas with curvilinear section $\Gamma_\phi$ by total variation.
Then we represent the pointwise supremum of $\mathbb{C}_{\Gamma_\phi}$ in terms of the two constructed copulas.

To begin with, for $\phi\in\Phi$ and $\Gamma_\phi\in\Delta_\phi$, we define functions $f_1, f_2\colon [0, 1]^2\to (-\infty, +\infty)$ as follows:
\begin{equation*}
f_1(x, y)=x-\frac{1}{2}\Big(\mathbb{V}^{x}_{\phi^{-1}(y)}(\hat{\Gamma}_{\phi})+\hat{\Gamma}_{\phi}(x)+\hat{\Gamma}_{\phi}(\phi^{-1}(y)) \Big)
\end{equation*}
and
\begin{equation*}
f_2(x, y)=y-\frac{1}{2}\Big(\mathbb{V}^{\phi^{-1}(y)}_{x}(\tilde{\Gamma}_{\phi})+\tilde{\Gamma}_{\phi}(x)+ \tilde{\Gamma}_{\phi}(\phi^{-1}(y))\Big).
\end{equation*}
Clearly, $f_1(0, 0)=f_2(0, 0)=0$ and $f_1(1, 1)=f_2(1, 1)=1$. Moreover, $f_1$ and $f_2$ are increasing, implying that they are
binary operations on $[0, 1]$.

\begin{prot}\label{prot-agop}
Both $f_1$ and $f_2$ are increasing.
\end{prot}
\begin{proof}
We only prove that $f_1$ is increasing since the proof concerning $f_2$ is similar.
For any $x_1, x_2, y\in[0, 1]$ with $x_1<x_2$, we have
\[
f_1(x_2, y)-f_1(x_1, y)=x_2-x_1-\frac{1}{2}\Big(\mathbb{V}_{x_1}^{x_2}(\hat{\Gamma}_\phi)+\hat{\Gamma}_\phi(x_2)-\hat{\Gamma}_\phi(x_1)\Big).
\]
It follows from
\[
\mathbb{V}_{x_1}^{x_2}(\hat{\Gamma}_\phi)\leq\mathbb{V}_{x_1}^{x_2}({\rm id})+ \mathbb{V}_{x_1}^{x_2}(\Gamma_\phi)=(x_2-x_1)+(\Gamma_\phi(x_2)-\Gamma_\phi(x_1))
\]
that $f_1(x_2, y)-f_1(x_1, y)\geq 0$, \emph{i.e.}, $f_1$ is increasing in the first variable.

For any $y_1, y_2, x\in[0, 1]$ with $y_1<y_2$, we have
\[
f_1(x, y_2)-f_1(x, y_1)=\frac{1}{2}\Big(\mathbb{V}_{\phi^{-1}(y_1)}^{\phi^{-1}(y_2)}(\hat{\Gamma}_\phi)-\hat{\Gamma}_\phi(\phi^{-1}(y_2))+\hat{\Gamma}_\phi(\phi^{-1}(y_1))\Big)\geq 0,
\]
\emph{i.e.}, $f_1$ is also increasing in the second variable.
\end{proof}

Before proving that both $f_1$ and $f_2$ are $1$-Lipschitz, we need two lemmas.
\begin{lemm}\label{lemm-evofGamm}
Let $\phi\in\Phi$ and $\Gamma_\phi\in\Delta_\phi$. Then the following inequalities hold for all $t_1, t_2\in[0, 1]$ with $t_1<t_2$:
\begin{itemize}\setlength{\itemindent}{10pt}
\item[{\rm (i)}] $|\hat{\Gamma}_\phi(t_2)-\hat{\Gamma}_\phi(t_1)|\leq 2(\phi(t_2)-\phi(t_1))+ (\hat{\Gamma}_\phi(t_2)-\hat{\Gamma}_\phi(t_1))$;
\item[{\rm (ii)}] $|\tilde{\Gamma}_\phi(t_2)-\tilde{\Gamma}_\phi(t_1)|\leq 2(t_2-t_1)+(\tilde{\Gamma}_\phi(t_2)-\tilde{\Gamma}_\phi(t_1))$;
\item[{\rm (iii)}] $|\hat{\Gamma}_\phi(t_2)-\hat{\Gamma}_\phi(t_1)|+|\tilde{\Gamma}_\phi(t_2)-\tilde{\Gamma}_\phi(t_1)|\leq (t_2-t_1)+(\phi(t_2)-\phi(t_1))$.
\end{itemize}
\end{lemm}
\begin{proof}
(i) We distinguish two possible cases.
\begin{itemize}
\item[\rm (a)] $\hat{\Gamma}_\phi(t_2)-\hat{\Gamma}_\phi(t_1)\geq 0$.

In this case, we have
\begin{eqnarray*}
|\hat{\Gamma}_\phi(t_2)-\hat{\Gamma}_\phi(t_1)|
&=&\hat{\Gamma}_\phi(t_2)-\hat{\Gamma}_\phi(t_1)\\
&\leq& 2(\phi(t_2)-\phi(t_1))+(\hat{\Gamma}_\phi(t_2)-\hat{\Gamma}_\phi(t_1)).
\end{eqnarray*}

\item[\rm (b)] $\hat{\Gamma}_\phi(t_2)-\hat{\Gamma}_\phi(t_1)<0$.

In this case, since $\Gamma_\phi(t_2)-\Gamma_\phi(t_1)\leq (t_2-t_1)+(\phi(t_2)-\phi(t_1))$, we have
$(\phi(t_2)-\phi(t_1))+(\hat{\Gamma}_\phi(t_2)-\hat{\Gamma}_\phi(t_1))\geq 0$. Hence, it holds that
\begin{eqnarray*}
|\hat{\Gamma}_\phi(t_2)-\hat{\Gamma}_\phi(t_1)|
&=&-(\hat{\Gamma}_\phi(t_2)-\hat{\Gamma}_\phi(t_1))\\
&\leq& 2(\phi(t_2)-\phi(t_1))+(\hat{\Gamma}_\phi(t_2)-\hat{\Gamma}_\phi(t_1)).
\end{eqnarray*}
\end{itemize}

(ii) The proof is similar to that of (i).

(iii) If both $\hat{\Gamma}_\phi(t_2)-\hat{\Gamma}_\phi(t_1)\geq 0$ and $\tilde{\Gamma}_\phi(t_2)-\tilde{\Gamma}_\phi(t_1)\geq 0$, then
\begin{eqnarray*}
&&|\hat{\Gamma}_\phi(t_2)-\hat{\Gamma}_\phi(t_1)|+|\tilde{\Gamma}_\phi(t_2)-\tilde{\Gamma}_\phi(t_1)|\\
&=&(t_2-t_1)+(\phi(t_2)-\phi(t_1))-2(\Gamma_\phi(t_2)-\Gamma_\phi(t_1))\\
&\leq& (t_2-t_1)+(\phi(t_2)-\phi(t_1)).
\end{eqnarray*}

If both $\hat{\Gamma}_\phi(t_2)-\hat{\Gamma}_\phi(t_1)< 0$ and $\tilde{\Gamma}_\phi(t_2)-\tilde{\Gamma}_\phi(t_1)< 0$, then it follows from $|\hat{\Gamma}_\phi(t_2)-\hat{\Gamma}_\phi(t_1)|\leq \phi(t_2)-\phi(t_1)$ and $|\tilde{\Gamma}_\phi(t_2)-\tilde{\Gamma}_\phi(t_1)|\leq t_2-t_1$ that the inequality holds.

If $\hat{\Gamma}_\phi(t_2)-\hat{\Gamma}_\phi(t_1)\geq 0$ and $\tilde{\Gamma}_\phi(t_2)-\tilde{\Gamma}_\phi(t_1)< 0$, then
\begin{eqnarray*}
|\hat{\Gamma}_\phi(t_2)-\hat{\Gamma}_\phi(t_1)|+|\tilde{\Gamma}_\phi(t_2)-\tilde{\Gamma}_\phi(t_1)|
&=& (t_2-t_1)-(\phi(t_2)-\phi(t_1))\\
&\leq& (t_2-t_1)+(\phi(t_2)-\phi(t_1)).
\end{eqnarray*}
Similarly, the inequality holds if $\hat{\Gamma}_\phi(t_2)-\hat{\Gamma}_\phi(t_1)<0$ and $\tilde{\Gamma}_\phi(t_2)-\tilde{\Gamma}_\phi(t_1)\geq 0$.
\end{proof}

The following estimations for the total variations of $\hat{\Gamma}_\phi$ and $\tilde{\Gamma}_\phi$ are crucial to upcoming results.
\begin{lemm}\label{lemm-evofTV}
Let $\phi\in\Phi$ and $\Gamma_\phi\in\Delta_\phi$. Then the following inequalities hold for all $t_1, t_2\in[0, 1]$ with $t_1<t_2$:
\begin{itemize}\setlength{\itemindent}{10pt}
\item[{\rm (i)}] $\mathbb{V}_{t_1}^{t_2}(\hat{\Gamma}_\phi)\leq 2(\phi(t_2)-\phi(t_1))+ (\hat{\Gamma}_\phi(t_2)-\hat{\Gamma}_\phi(t_1))$;
\item[{\rm (ii)}] $\mathbb{V}_{t_1}^{t_2}(\tilde{\Gamma}_\phi)\leq 2(t_2-t_1)+(\tilde{\Gamma}_\phi(t_2)-\tilde{\Gamma}_\phi(t_1))$;
\item[{\rm (iii)}] $\mathbb{V}_{t_1}^{t_2}(\hat{\Gamma}_\phi)+ \mathbb{V}_{t_1}^{t_2}(\tilde{\Gamma}_\phi)\leq (t_2-t_1)+ (\phi(t_2)-\phi(t_1))$.
\end{itemize}
\end{lemm}
\begin{proof}
(i) For any $\{x_i\}_{i=0}^n\in\mathcal P[t_1, t_2]$, it follows from Lemma~\ref{lemm-evofGamm}(i) that
\begin{eqnarray*}
\sum_{i=1}^n|\hat{\Gamma}_\phi(x_i)-\hat{\Gamma}_\phi(x_{i-1})|
&\leq& \sum_{i=1}^n\bigg(2(\phi(x_{i})-\phi(x_{i-1}))+(\hat{\Gamma}_\phi(x_i)-\hat{\Gamma}_\phi(x_{i-1}))\bigg)\\
&=& 2(\phi(t_2)-\phi(t_1))+(\hat{\Gamma}_\phi(t_2)-\hat{\Gamma}_\phi(t_1)).
\end{eqnarray*}
Hence, $\mathbb{V}_{t_1}^{t_2}(\hat{\Gamma}_\phi)\leq 2(\phi(t_2)-\phi(t_1))+ (\hat{\Gamma}_\phi(t_2)-\hat{\Gamma}_\phi(t_1))$.

(ii) The proof is similar to that of (i).

(iii) For any $\epsilon>0$, there exists $\{x_i\}_{i=0}^n\in\mathcal P[t_1, t_2]$ such that
\[
\mathbb{V}_{t_1}^{t_2}(\hat{\Gamma}_\phi)<\sum_{i=1}^n|\hat{\Gamma}_\phi(x_i)-\hat{\Gamma}_\phi(x_{i-1})|+\epsilon
\]
and
\[
\mathbb{V}_{t_1}^{t_2}(\tilde{\Gamma}_\phi)<\sum_{i=1}^n|\tilde{\Gamma}_\phi(x_i)-\tilde{\Gamma}_\phi(x_{i-1})|+\epsilon~.
\]
It follows from Lemma~\ref{lemm-evofGamm}(iii) that
\begin{eqnarray*}
&&\sum_{i=1}^n|\hat{\Gamma}_\phi(x_i)-\hat{\Gamma}_\phi(x_{i-1})|+ \sum_{i=1}^n|\tilde{\Gamma}_\phi(x_i)-\tilde{\Gamma}_\phi(x_{i-1})|\\
&=&\sum_{i=1}^n\Big(|\hat{\Gamma}_\phi(x_i)-\hat{\Gamma}_\phi(x_{i-1})|+|\tilde{\Gamma}_\phi(x_i)-\tilde{\Gamma}_\phi(x_{i-1})|\Big)\\
&\leq& \sum_{i=1}^n\Big((x_i-x_{i-1})+(\phi(x_{i})-\phi(x_{i-1}))\Big)
= (t_2-t_1)+(\phi(t_2)-\phi(t_1)).
\end{eqnarray*}
Hence, $\mathbb{V}_{t_1}^{t_2}(\hat{\Gamma}_\phi)+\mathbb{V}_{t_1}^{t_2}(\tilde{\Gamma}_\phi)<(t_2-t_1)+(\phi(t_2)-\phi(t_1))+2\epsilon$.
Therefore, the conclusion holds due to the arbitrariness of $\epsilon$.
\end{proof}

\begin{prot}\label{prot-lip}
Both $f_1$ and $f_2$ are $1$-Lipschitz.
\end{prot}
\begin{proof}
For any $x_1, x_2, y\in[0, 1]$ with $x_1<x_2$, we have
\[
f_1(x_2, y)-f_1(x_1, y)=(x_2-x_1)-\frac{1}{2}\Big(\mathbb{V}_{x_1}^{x_2}(\hat{\Gamma}_\phi)+\hat{\Gamma}_\phi(x_2)-\hat{\Gamma}_\phi(x_1)\Big)\leq x_2-x_1.
\]
Since $f_1$ is increasing, it follows that $f_1$ is $1$-Lipschitz w.r.t. the first variable.

For any $y_1, y_2, x\in[0, 1]$ with $y_1<y_2$, it follows from Lemma~\ref{lemm-evofTV}(i) that
\begin{eqnarray*}
f_1(x, y_2)-f_1(x, y_1)
&=&\frac{1}{2}\Big(\mathbb{V}_{\phi^{-1}(y_1)}^{\phi^{-1}(y_2)}(\hat{\Gamma}_\phi)-\hat{\Gamma}_\phi(\phi^{-1}(y_2))+\hat{\Gamma}_\phi(\phi^{-1}(y_1))\Big)\\
&\leq& y_2-y_1.
\end{eqnarray*}
Since $f_1$ is increasing, $f_1$ is also $1$-Lipschitz w.r.t. the second variable.

Similarly, we can show that $f_2$ is $1$-Lipschitz.
\end{proof}

Based on $f_1$ and $f_2$, we define functions $C_1, C_2\colon [0, 1]^2\to [0, 1]$ as follows:
\begin{equation*}
C_1(x, y)=\min\{x, y, f_1(x, y)\}
\end{equation*}
and
\begin{equation*}
C_2(x, y)=\min\{x, y, f_2(x, y)\}.
\end{equation*}
Interestingly, $C_1$ and $C_2$ are copulas with curvilinear section $\Gamma_\phi$.
\begin{thrm}
Let $\phi\in\Phi$ and $\Gamma_\phi\in\Delta_\phi$. Then $C_1, C_2\in\mathbb{C}_{\Gamma_\phi}$.
\end{thrm}
\begin{proof}
We only prove that $C_1\in \mathbb{C}_{\Gamma_\phi}$ since the proof of $C_2\in \mathbb{C}_{\Gamma_\phi}$ is similar.
It is clear that $C_{1}(x, \phi(x))=\Gamma_\phi(x)$ for all $x\in[0, 1]$. It then remains to prove that $C_1$ is a copula.

\emph{Boundary condition:} For any $x\in[0, 1]$, we have
\[
\mathbb{V}^{x}_{\phi^{-1}(1)}(\hat{\Gamma}_{\phi})+\hat{\Gamma}_{\phi}(x)+\hat{\Gamma}_{\phi}(\phi^{-1}(1))
=\hat{\Gamma}_{\phi}(x)-\mathbb{V}^{1}_{x}(\hat{\Gamma}_{\phi})\leq 0.
\]
Hence, $f_1(x, 1)\geq x$, and thus $C_1(x, 1)=x$. It follows from Lemma~\ref{lemm-evofTV}(i) that
\[
\mathbb{V}^{1}_{\phi^{-1}(x)}(\hat{\Gamma}_{\phi})+\hat{\Gamma}_{\phi}(1)+\hat{\Gamma}_{\phi}(\phi^{-1}(x))\leq 2(\phi(1)-x)+2\hat{\Gamma}_{\phi}(1)=2(1-x).
\]
Hence, $f_1(1, x)\geq x$, and thus $C_1(1, x)=x$. Since $f_1$ is increasing, $C_1$ is also increasing, which implies that
$C_1(0, x)=C_1(x, 0)=0$.

\emph{$2$-increasingness:} Note that any rectangle $R=[a, b]\times [c, d]$ can be partitioned into at most three non-overlapping rectangles $R_i$ and $V(R)$ is the sum of all $V(R_i)$,
where each $R_i$ belongs to one of the three types of rectangles:

\noindent Type 1: $R\subset\{(x, y)\mid y\leq\phi(x)\}$, \emph{i.e.}, $R$ is below the curve $\{(x, \phi(x))\mid x\in [0, 1]\}$;

\noindent Type 2: $R\subset\{(x, y)\mid y\geq\phi(x)\}$, \emph{i.e.}, $R$ is above the curve $\{(x, \phi(x))\mid x\in [0, 1]\}$;

\noindent Type 3: $R=[x_1, x_2]\times [\phi(x_1), \phi(x_2)]$, \emph{i.e.}, $R$ has two vertices on the curve $\{(x, \phi(x))\mid x\in [0, 1]\}$.

\noindent We only need to prove that $V(R)\geq 0$ for each of the three types of rectangles $R=[x_1, x_2]\times [y_1, y_2]$.

\noindent Type 1: $R=[x_1, x_2]\times[y_1, y_2]\subset\{(x, y)\mid y\leq\phi(x)\}$.

In this case, we have $C_{1}(x, y)=\min\{y, f_1(x, y)\}$ for any $(x, y)\in R$.

If $C_1(x_1, y_1)=y_1$, then it follows from the increasingness of $C_1$ that $V(R)=C_1(x_2, y_2)-C_1(x_1, y_2)\geq 0$.

For the case that $C_1(x_1, y_1)=f_1(x_1, y_1)$, there are two subcases to be considered:
\begin{itemize}
\item[\rm (a)] $C_1(x_2, y_2)=y_2$.

In this subcase, we have
\[
V(R)\geq f_1(x_1, y_1)+y_2-y_1-f_1(x_1, y_2)\geq 0,
\]
where the last inequality is due to the fact that $f_1$ is $1$-Lipschitz.

\item[\rm (b)] $C_1(x_2, y_2)=f_1(x_2, y_2)$.

In this subcase, we have
\[
V(R)\geq f_1(x_1, y_1)+f_1(x_2, y_2)-f_1(x_1, y_2)-f_1(x_2, y_1)=0.
\]
\end{itemize}

\noindent Type 2: $R=[x_1, x_2]\times[y_1, y_2]\subset\{(x, y)\mid y\geq\phi(x)\}$.

If $C_1(x_1, y_1)=x_1$ or $C_1(x_1, y_1)=y_1$, then the increasingness of $C_1$ implies that $V(R)\geq 0$. So we suppose that
$C_1(x_1, y_1)=f_1(x_1, y_1)$. There are three subcases to be considered:
\begin{itemize}
\item[\rm (a)] $C_1(x_2, y_2)=x_2$.

In this subcase, we have
\[
V(R)\geq f_1(x_1, y_1)+x_2-x_1-f_1(x_2, y_1)\geq 0.
\]

\item[\rm (b)] $C_1(x_2, y_2)=y_2$.

In this subcase, we have
\[
V(R)\geq f_1(x_1, y_1)+y_2-y_1-f_1(x_1, y_2)\geq 0.
\]

\item[\rm (c)] $C_1(x_2, y_2)=f_1(x_2, y_2)$.

In this subcase, we have
\[
V(R)\geq f_1(x_1, y_1)+f_1(x_2, y_2)-f_1(x_1, y_2)-f_1(x_2, y_1)=0.
\]
\end{itemize}

\noindent Type 3: $R=[x_1, x_2]\times [\phi(x_1), \phi(x_2)]$.

In this case, we have $C_1(x_1, \phi(x_1))=\Gamma_\phi(x_1)$ and $C_1(x_2, \phi(x_2))=\Gamma_\phi(x_2)$. Hence,
\begin{eqnarray*}
V(R)&\geq& \Gamma_\phi(x_1)+\Gamma_\phi(x_2)-f_1(x_1, \phi(x_2))-f_1(x_2, \phi(x_1))\\
&=& \Gamma_\phi(x_1)+\Gamma_\phi(x_2)-x_1+\frac{1}{2}\Big(\mathbb{V}^{x_1}_{x_2}(\hat{\Gamma}_{\phi})+\hat{\Gamma}_{\phi}(x_1)+\hat{\Gamma}_{\phi}(x_2)\Big)\\
&&- x_2+ \frac{1}{2}\Big(\mathbb{V}^{x_2}_{x_1}(\hat{\Gamma}_{\phi})+\hat{\Gamma}_{\phi}(x_1)+\hat{\Gamma}_{\phi}(x_2)\Big)\\
&=& \Gamma_\phi(x_1)+\Gamma_\phi(x_2)-\Gamma_\phi(x_1)-\Gamma_\phi(x_2)=0.
\end{eqnarray*}
We conclude that $C_1\in \mathbb{C}_{\Gamma_\phi}$.
\end{proof}

The $\phi$-splice $C_2\MySs_\phi C_1$~\cite{JDHD21} of the two copulas $C_1$ and $C_2$ is defined by
\[
(C_2\MySs_\phi C_1)(x, y)= \left \{
        \begin {array}{ll}
         C_1(x, y)
                  &\quad \text{if\ \ $y\leq\phi(x)$}  \\[2mm]
         C_2(x, y)
                  &\quad \text{if\ \ $y\geq\phi(x)$.}
        \end {array}
       \right.
\]
It is known from~\cite{XWZJ23} that $C_2\MySs_\phi C_1$ is a quasi-copula.
Now we prove that $C_2\MySs_\phi C_1$ is nothing else but the pointwise supremum of $\mathbb{C}_{\Gamma_\phi}$.

\begin{thrm}
Let $\phi\in\Phi$ and $\Gamma_\phi\in\Delta_\phi$. Then it holds that

\centerline{$(C_2\MySs_\phi C_1)(x, y)=\sup\{C(x, y)\mid C\in\mathbb{C}_{\Gamma_\phi}\}$ for all $(x, y)\in[0, 1]^2$.}
\end{thrm}
\begin{proof}
Consider $(x, y)\in[0, 1]^2$. We only prove the equality for the case that $y\leq\phi(x)$ (the proof is similar if $y\geq\phi(x)$).
As $C_1\in\mathbb{C}_{\Gamma_\phi}$, we need to show that $C(x, y)\leq C_1(x, y)$ for all $C\in\mathbb{C}_{\Gamma_\phi}$, or equivalently,
\[C(x, y)\leq x-\frac{1}{2}\Big(\mathbb{V}^{x}_{\phi^{-1}(y)}(\hat{\Gamma}_{\phi})+\hat{\Gamma}_{\phi}(x)+\hat{\Gamma}_{\phi}(\phi^{-1}(y)) \Big).\]
Indeed, for any $\{x_i\}_{i=0}^{2m}\in\mathcal P[\phi^{-1}(y), x]$, we have
\begin{eqnarray*}
0
&\leq& V_C([x_{2m-1}, x]\times[y, 1])+V_C([0, x_1]\times[y, \phi(x_1)])+\sum^{m-1}_{k=1}V_C([x_{2k-1}, x_{2k+1}]\times[y, \phi(x_{2k+1})]) \\
&& +\sum^{m-1}_{k=1}V_C([0, x_{2k}]\times[\phi(x_{2k}), \phi(x_{2k+1})])+\sum^{m-1}_{k=1}V_C([x_{2k-1}, x_{2k}]\times[\phi(x_{2k+1}), 1])\\
&=& x-C(x, y)+\sum^{m-1}_{k=1}\hat{\Gamma}_\phi(x_{2k})-\sum^{m}_{k=1}\hat{\Gamma}_\phi(x_{2k-1}),
\end{eqnarray*}
which implies that
\[\sum^{m}_{k=1}\hat{\Gamma}_\phi(x_{2k-1})-\sum^{m-1}_{k=1}\hat{\Gamma}_\phi(x_{2k})\leq x-C(x, y).\]
Hence, it follows from Remark~\ref{remk-variation} that
\[\mathbb{V}_{\phi^{-1}(y)}^{x}(\hat{\Gamma}_\phi)\leq 2(x-C(x, y))-\hat{\Gamma}_\phi(\phi^{-1}(y))-\hat{\Gamma}_\phi(x),\]
which is equivalent to the inequality as desired:
\[C(x, y)\leq x-\frac{1}{2}\Big(\mathbb{V}^{x}_{\phi^{-1}(y)}(\hat{\Gamma}_{\phi})+\hat{\Gamma}_{\phi}(x)+\hat{\Gamma}_{\phi}(\phi^{-1}(y)) \Big).\]
This completes the proof.
\end{proof}

We illustrate the pointwise supremum as follows.
\begin{exa}\label{exa-demon}
Let $\phi\in\Phi$ be defined by $\phi(t)=t^2$ and $\Gamma_\phi\in\Delta_\phi$ be defined by $\Gamma_\phi(t)=t^3$.
Then $C_1$ and $C_2$ are easily computed as follows:
\[
C_1(x, y)= \left \{
        \begin {array}{ll}
        \min\{x^3, y\}
                  &\quad \text{if\ \ $x, \sqrt{y}\leq\dfrac{\sqrt{3}}{3}$}  \\[2mm]
        \min\{x-\sqrt{y}+y\sqrt{y}, y\}
                  &\quad \text{if\ \ $x, \sqrt{y}\geq\dfrac{\sqrt{3}}{3}$}  \\[2mm]
        \min\{x-\dfrac{2\sqrt{3}}{9}, y\}
                  &\quad \text{if\ \ $\sqrt{y}\leq\dfrac{\sqrt{3}}{3}\leq x$}  \\[2mm]
        \min\{x, x^3-\sqrt{y}+y\sqrt{y}+\dfrac{2\sqrt{3}}{9}\}
                  &\quad \text{if\ \ $x\leq\dfrac{\sqrt{3}}{3}\leq \sqrt{y}$;}
        \end {array}
       \right.
\]
\[
C_2(x, y)= \left \{
        \begin {array}{ll}
        \min\{x, y\sqrt{y}\}
                  &\quad \text{if\ \ $x, \sqrt{y}\leq\dfrac{2}{3}$}  \\[2mm]
        \min\{x, -x^2+x^3+y\}
                  &\quad \text{if\ \ $x, \sqrt{y}\geq\dfrac{2}{3}$}  \\[2mm]
        \min\{-x^2+x^3+y\sqrt{y}+\dfrac{4}{27}, y\}
                  &\quad \text{if\ \ $\sqrt{y}\leq\dfrac{2}{3}\leq x$}  \\[2mm]
        \min\{x, y-\dfrac{4}{27}\}
                  &\quad \text{if\ \ $x\leq\dfrac{2}{3}\leq \sqrt{y}$.}
        \end {array}
       \right.
\]
The 3D plots of $C_1, C_2$ and $C_2\MySs_\phi C_1$ are visualized in Figure~\ref{fig-1}.
Note that $C_2\MySs_\phi C_1$ is not a copula as $V([\dfrac{1}{4}, \dfrac{1}{3}]\times[\dfrac{1}{16}, \dfrac{1}{9}])<0$.
\begin{figure}
\centering
\scalebox{0.2}[0.2]{
  \includegraphics{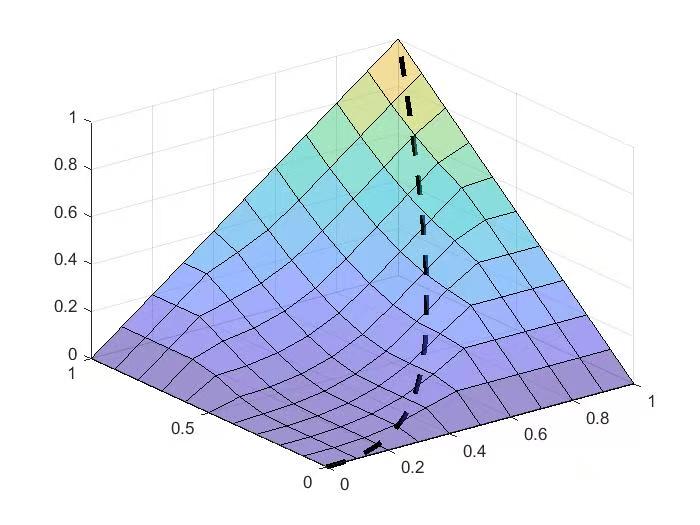}\includegraphics{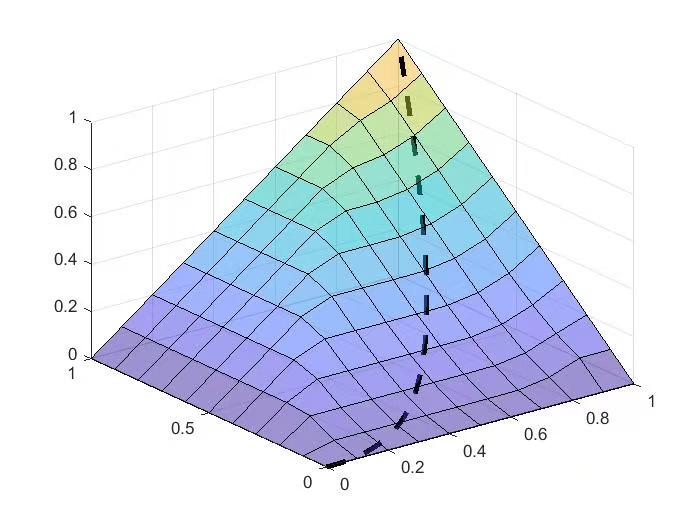}\includegraphics{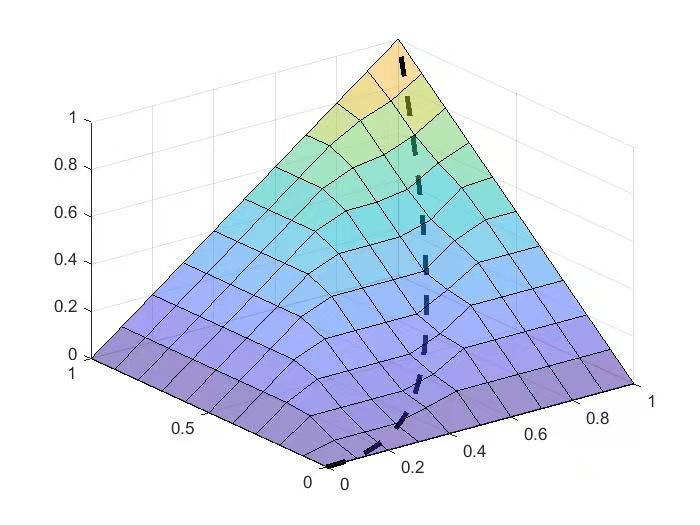}}
  \caption{3D plots of $C_1$ (left), $C_2$ (middle), and the $\phi$-splice of $C_1$ and $C_2$ (right) in Example~\ref{exa-demon}.}
  \label{fig-1}
\end{figure}
\end{exa}

\begin{remk}
For any quasi-copula $Q$ with $Q(x_0, y_0)=x_0=M(x_0, y_0)$ at some point $(x_0, y_0)\in[0, 1]^2$, it follows from (Q1) and (Q2)
that $Q(x, y)=x=M(x, y)$ for all $(x, y)\in[0, x_0]\times[y_0, 1]$. Similarly, if $Q(x_0, y_0)=y_0=M(x_0, y_0)$ at some point
$(x_0, y_0)\in[0, 1]^2$, then $Q(x, y)=y=M(x, y)$ for all $(x, y)\in[x_0, 1]\times[0, y_0]$.

According to the above observation, we can deduce some points at which the quasi-copula $C_2\MySs_\phi C_1$ behaves as $M$ from the set $A=\{t\in[0, 1]\mid \Gamma_\phi(t)=\min\{t, \phi(t)\}\}$. However, $C_2\MySs_\phi C_1$ may also behave as $M$ at other points. For instance, $A=\{0, 1\}$ in Example~\ref{exa-demon} can only imply by the above observation that $C_2\MySs_\phi C_1$ behaves as $M$ on the boundary of the unit square $[0, 1]^2$, but there are other points such as $(\dfrac{1}{3}, \dfrac{1}{81})$ at which
$C_2\MySs_\phi C_1$ behaves as $M$.

In conclusion, finding all the points at which $C_2\MySs_\phi C_1$ behaves as $M$ could be an interesting question for future work.
\end{remk}

\section{Characterization for the pointwise supremum of $\mathbb{C}_{\Gamma_\phi}$ to be a copula}\label{Sec-ChaCop}

In last section we have proved that the pointwise supremum of $\mathbb{C}_{\Gamma_\phi}$ is the $\phi$-splice $C_2\MySs_\phi C_1$ of the two copulas $C_1$ and $C_2$. In this section, we characterize when $C_2\MySs_\phi C_1$ is a copula.

Clearly, $C_2\MySs_\phi C_1$ can be reformulated by
\[
(C_2\MySs_\phi C_1)(x, y)= \left \{
        \begin {array}{ll}
        \min\bigg\{y, x-\dfrac{1}{2}\Big(\mathbb{V}^{x}_{\phi^{-1}(y)}(\hat{\Gamma}_{\phi})+\hat{\Gamma}_{\phi}(x)+\hat{\Gamma}_{\phi}(\phi^{-1}(y)) \Big)\bigg\}
                  &\quad \text{if\ \ $y\leq\phi(x)$}  \\[2mm]
         \min\bigg\{x, y-\dfrac{1}{2}\Big(\mathbb{V}^{\phi^{-1}(y)}_{x}(\tilde{\Gamma}_{\phi})+\tilde{\Gamma}_{\phi}(x)+ \tilde{\Gamma}_{\phi}(\phi^{-1}(y))\Big)\bigg\}
                  &\quad \text{if\ \ $y\geq\phi(x)$.}
        \end {array}
       \right.
\]
Since $C_1$ and $C_2$ are copulas, $C_2\MySs_\phi C_1$ is a copula if and only if $V(R)\geq 0$ for each rectangle $R=[x_1, x_2]\times [\phi(x_1), \phi(x_2)]$ (see~\cite{JDHD21}). For such a rectangle $R$, we have
\begin{eqnarray*}
V(R)&=& \Gamma_\phi(x_1)-\min\left\{x_1, \phi(x_2)-\frac{1}{2}\Big(\mathbb{V}_{x_1}^{x_2}(\tilde{\Gamma}_\phi)+\tilde{\Gamma}_\phi(x_1)+\tilde{\Gamma}_\phi(x_2)\Big)\right\}\\
&&-\min\left\{\phi(x_1), x_2-\frac{1}{2}\Big(\mathbb{V}_{x_1}^{x_2}(\hat{\Gamma}_\phi)+\hat{\Gamma}_\phi(x_1)+\hat{\Gamma}_\phi(x_2)\Big)\right\}+\Gamma_\phi(x_2)\\
&=& \max\left\{\Gamma_\phi(x_1)-x_1-\phi(x_1)+\Gamma_\phi(x_2), \frac{1}{2}\Big(\mathbb{V}_{x_1}^{x_2}(\hat{\Gamma}_\phi)-\hat{\Gamma}_\phi(x_1)-\hat{\Gamma}_\phi(x_2)\Big),\right.\\
&& \frac{1}{2}\Big(\mathbb{V}_{x_1}^{x_2}(\tilde{\Gamma}_\phi)-\tilde{\Gamma}_\phi(x_1)-\tilde{\Gamma}_\phi(x_2)\Big), \left.\frac{1}{2}\Big(\mathbb{V}_{x_1}^{x_2}(\hat{\Gamma}_\phi)+\mathbb{V}_{x_1}^{x_2}(\tilde{\Gamma}_\phi)-(x_2-x_1)-(\phi(x_2)-\phi(x_1))\Big)\right\}.
\end{eqnarray*}
Hence, $C_2\MySs_\phi C_1$ is a copula if and only if, for any
$(x_1, x_2)\in[0, 1]^2$ with $x_1<x_2$, at least one of the following four inequalities holds:
\begin{equation}\label{Condition-4}
\Gamma_\phi(x_1)+\Gamma_\phi(x_2)\geq x_1+\phi(x_1)
\end{equation}
\begin{equation}\label{Condition-1}
\mathbb{V}_{x_1}^{x_2}(\hat{\Gamma}_\phi)\geq\hat{\Gamma}_\phi(x_1)+\hat{\Gamma}_\phi(x_2)
\end{equation}
\begin{equation}\label{Condition-2}
\mathbb{V}_{x_1}^{x_2}(\tilde{\Gamma}_\phi)\geq\tilde{\Gamma}_\phi(x_1)+\tilde{\Gamma}_\phi(x_2)
\end{equation}
\begin{equation}\label{Condition-3}
\mathbb{V}_{x_1}^{x_2}(\hat{\Gamma}_\phi)+\mathbb{V}_{x_1}^{x_2}(\tilde{\Gamma}_\phi)\geq(x_2-x_1)+(\phi(x_2)-\phi(x_1)).
\end{equation}
Lemma~\ref{lemm-evofTV}(iii) ensures that the inequality~\eqref{Condition-3} is actually an equality.

If $C_2\MySs_\phi C_1$ is a copula, then inequality~\eqref{Condition-4} implies either~\eqref{Condition-1} or~\eqref{Condition-2}. Before proving this implication, we need the following result.

\begin{lemm}\label{lemm-condition}
For $(x_1, x_2)\in[0, 1]^2$ with $x_1<x_2$, the following results hold:
\begin{itemize}\setlength{\itemindent}{10pt}
\item[{\rm (i)}] If $\mathbb{V}_{x_1}^{x_2}(\hat{\Gamma}_\phi)\geq\hat{\Gamma}_\phi(x_1)+\hat{\Gamma}_\phi(x_2)$, then  $\mathbb{V}_{x_1}^{x}(\hat{\Gamma}_\phi)\geq\hat{\Gamma}_\phi(x_1)+\hat{\Gamma}_\phi(x)$ for any $x\in [x_2, 1]$;
\item[{\rm (ii)}] If $\mathbb{V}_{x_1}^{x_2}(\tilde{\Gamma}_\phi)\geq\tilde{\Gamma}_\phi(x_1)+\tilde{\Gamma}_\phi(x_2)$, then $\mathbb{V}_{x_1}^{x}(\tilde{\Gamma}_\phi)\geq\tilde{\Gamma}_\phi(x_1)+\tilde{\Gamma}_\phi(x)$ for any $x\in [x_2, 1]$.
\end{itemize}
\end{lemm}
\begin{proof}
(i) For any $x\in [x_2, 1]$, we have
\begin{eqnarray*}
\mathbb{V}_{x_1}^{x}(\hat{\Gamma}_\phi)-\hat{\Gamma}_\phi(x_1)-\hat{\Gamma}_\phi(x)
&=&\Big(\mathbb{V}_{x_1}^{x_2}(\hat{\Gamma}_\phi)-\hat{\Gamma}_\phi(x_1)-\hat{\Gamma}_\phi(x_2)\Big)\\
&&+\Big(\mathbb{V}_{x_2}^{x}(\hat{\Gamma}_\phi)-\hat{\Gamma}_\phi(x)+\hat{\Gamma}_\phi(x_2)\Big).
\end{eqnarray*}
It is clear that $\mathbb{V}_{x_2}^{x}(\hat{\Gamma}_\phi)-\hat{\Gamma}_\phi(x)+\hat{\Gamma}_\phi(x_2)\geq 0$, together with the assumption implying that
$\mathbb{V}_{x_1}^{x}(\hat{\Gamma}_\phi)-\hat{\Gamma}_\phi(x_1)-\hat{\Gamma}_\phi(x)\geq 0$.

(ii) The proof is similar to that of (i).
\end{proof}

\begin{prot}
Let $C_2\MySs_\phi C_1$ be a copula and the inequality~\eqref{Condition-4} hold for some $(x_1, x_2)\in[0, 1]^2$ with $x_1<x_2$.
Then either~\eqref{Condition-1} or~\eqref{Condition-2} holds.
\end{prot}
\begin{proof}
Suppose that neither~\eqref{Condition-1} nor~\eqref{Condition-2} holds at $(x_1, x_2)$, which implies that $\hat{\Gamma}_\phi(x_1)>0$ and $\tilde{\Gamma}_\phi(x_1)>0$, and thus $\Gamma_\phi(x_1)<\min\{x_1, \phi(x_1)\}$. Stipulate
\[
t^*=\sup\{t\mid t\in [x_1, x_2]\mbox{~and~}\Gamma_\phi(t)+\Gamma_\phi(x_1)-x_1-\phi(x_1)<0\}.
\]
Then $t^*\in\,]x_1, x_2]$ and $\Gamma_\phi(t^*)+\Gamma_\phi(x_1)-x_1-\phi(x_1)=0$. It follows from Lemma~\ref{lemm-condition} that

\centerline{$\mathbb{V}_{x_1}^{t^*}(\hat{\Gamma}_\phi)-\hat{\Gamma}_\phi(x_1)-\hat{\Gamma}_\phi(t^*)< 0$ and $\mathbb{V}_{x_1}^{t^*}(\tilde{\Gamma}_\phi)-\tilde{\Gamma}_\phi(x_1)-\tilde{\Gamma}_\phi(t^*)< 0$.}

\noindent Hence, we have
\begin{eqnarray*}
&&\mathbb{V}_{x_1}^{t^*}(\hat{\Gamma}_\phi)+\mathbb{V}_{x_1}^{t^*}(\tilde{\Gamma}_\phi)-(t^*-x_1)-(\phi(t^*)-\phi(x_1))\\
&=&\Big(\mathbb{V}_{x_1}^{t^*}(\hat{\Gamma}_\phi)-\hat{\Gamma}_\phi(x_1)-\hat{\Gamma}_\phi(t^*)\Big)
+\Big(\mathbb{V}_{x_1}^{t^*}(\tilde{\Gamma}_\phi)-\tilde{\Gamma}_\phi(x_1)-\tilde{\Gamma}_\phi(t^*)\Big)\\
&& -2\Big(\Gamma_\phi(t^*)+\Gamma_\phi(x_1)-x_1-\phi(x_1)\Big)\\
&<& 0.
\end{eqnarray*}
It is easy to see that there exists $t_0\in\,]x_1, t^*[\,$ such that
\[\mathbb{V}_{x_1}^{t_0}(\hat{\Gamma}_\phi)+\mathbb{V}_{x_1}^{t_0}(\tilde{\Gamma}_\phi)-(t_0-x_1)-(\phi(t_0)-\phi(x_1))<0.\]
Therefore, neither of the four inequalities~\eqref{Condition-4}--\eqref{Condition-3} holds at $(x_1, t_0)$, a contradiction with
the assumption that $C_2\MySs_\phi C_1$ is a copula.
\end{proof}

Due to the above proposition, we obtain a simplified characterization when $C_2\MySs_\phi C_1$ is a copula.
\begin{thrm}\label{thrm-CbeCopu}
$C_2\MySs_\phi C_1$ is a copula if and only if, for any $(x_1, x_2)\in[0, 1]^2$ with $x_1<x_2$, at least one of the three inequalities \eqref{Condition-1}--\eqref{Condition-3} holds.
\end{thrm}

As an application of the characterization, we illustrate when $C_2\MySs_\phi C_1$ is a copula.

\begin{exa}\label{Exa-illustrative}
Let $\phi\in\Phi$ and $\Gamma_\phi\in\Delta_\phi$ be determined by a family of pairwisely disjoint open subintervals
$\{\,]a_i, b_i[\,\}_{i\in I}$ of $[0, 1]$ (see~\cite{OSZ25}), \emph{i.e.},
\[
\Gamma_\phi(t) =\left \{
        \begin {array}{ll}
         \min\{a_i, \phi(a_i)\}
                  &\quad \text{\rm if\ \ $t\in \,]a_i, u_i^*]$} \\[2mm]
           \phi(t)+t-\max\{b_i, \phi(b_i)\}
                  &\quad \text{\rm if\ \ $t\in [u_i^*, b_i[\,$} \\[2mm]
         \min\{t, \phi(t)\}
                  &\quad {\rm otherwise,}
        \end {array}
       \right.
\]
where $u_i^*\in\,]a_i, b_i[\,$ is the unique solution to the equation: $\phi(t)+t=\min\{a_i, \phi(a_i)\}+\max\{b_i, \phi(b_i)\}$. For any $(x_1, x_2)\in[0, 1]^2$ with $x_1<x_2$, we can show that at least one of the three inequalities \eqref{Condition-1}--\eqref{Condition-3} holds by considering two cases:
\begin{itemize}
\item[\rm (i)] If there exists $t\in [x_1, x_2]$ such that $\Gamma_\phi(t)=t$, \emph{i.e.}, $\hat{\Gamma}_\phi(t)=0$, then
\[
\mathbb{V}_{x_1}^{x_2}(\hat{\Gamma}_\phi)-\hat{\Gamma}_\phi(x_1)-\hat{\Gamma}_\phi(x_2)\geq -2\hat{\Gamma}_\phi(t)=0.
\]
That is to say, \eqref{Condition-1} holds. Similarly, if $\Gamma_\phi(t)=\phi(t)$ for some $t\in [x_1, x_2]$, then \eqref{Condition-2} holds.

\item[\rm (ii)] Suppose that $[x_1, x_2]\subset \,]a_i, b_i[\,$ for some $i\in I$. If $x_1, x_2\leq u_i^*$, then both $\hat{\Gamma}_\phi(t)$ and $\tilde{\Gamma}_\phi(t)$ are increasing on $[x_1, x_2]$, which implies that
\[
\mathbb{V}_{x_1}^{x_2}(\hat{\Gamma}_\phi)+\mathbb{V}_{x_1}^{x_2}(\tilde{\Gamma}_\phi)=(x_2-x_1)+(\phi(x_2)-\phi(x_1)),
\]
\emph{i.e.}, \eqref{Condition-3} holds. Similarly, if $x_1, x_2\geq u_i^*$, then \eqref{Condition-3} also holds.
For the last subcase that $x_1<u_i^*<x_2$, we have
\begin{eqnarray*}
\mathbb{V}_{x_1}^{x_2}(\hat{\Gamma}_\phi)+\mathbb{V}_{x_1}^{x_2}(\tilde{\Gamma}_\phi)
&=&\Big(\mathbb{V}_{x_1}^{u_i^*}(\hat{\Gamma}_\phi)+\mathbb{V}_{x_1}^{u_i^*}(\hat{\Gamma}_\phi)\Big)
+\Big(\mathbb{V}_{u_i^*}^{x_2}(\tilde{\Gamma}_\phi)+\mathbb{V}_{u_i^*}^{x_2}(\tilde{\Gamma}_\phi)\Big)\\
&=&\Big(u_i^*-x_1+\phi(u_i^*)-\phi(x_1)\Big)+\Big(x_2-u_i^*+\phi(x_2)-\phi(u_i^*)\Big)\\
&=&(x_2-x_1)+(\phi(x_2)-\phi(x_1)),
\end{eqnarray*}
\emph{i.e.}, \eqref{Condition-3} also holds.
\end{itemize}
\end{exa}

There exist $\phi\in\Phi$ and $\Gamma_\phi\in\Delta_\phi$ that is not of the form as in Example~\ref{Exa-illustrative} such that $C_2\MySs_\phi C_1$ is a copula.
\begin{exa}\label{Exa-itisacopula}
Let $\phi\in\Phi$ be defined by $\phi(t)=t^2$ and $\Gamma_\phi\in\Delta_\phi$ be defined by
\[
\Gamma_\phi(t) =\left \{
        \begin {array}{ll}
         0
                  &\quad \text{\rm if\ \ $t\in [0, \dfrac{1}{3}]$} \\[2mm]
           t+t^2-\dfrac{4}{9}
                  &\quad \text{\rm if\ \ $t\in\,]\dfrac{1}{3}, \dfrac{2}{5}]\,$} \\[2mm]
           \dfrac{26}{225}
                  &\quad \text{\rm if\ \ $t\in\,]\dfrac{2}{5}, t_0]\,$} \\[2mm]
           t+t^2-1
                  &\quad \text{\rm if\ \ $t\in\,]t_0, 1]\,$},
        \end {array}
       \right.
\]
where $\dfrac{2}{5}<t_0<1$ and $t_0+t_0^2-1=\dfrac{26}{225}$. It is an easy exercise to verify that, for any $(x_1, x_2)\in[0, 1]^2$ with $x_1<x_2$, it holds that $\mathbb{V}_{x_1}^{x_2}(\hat{\Gamma}_\phi)+\mathbb{V}_{x_1}^{x_2}(\tilde{\Gamma}_\phi)=(x_2-x_1)+(x^2_2-x^2_1)$, \emph{i.e.}, \eqref{Condition-3} always holds. Hence, $C_2\MySs_\phi C_1$ is a copula that is visualized in Figure~\ref{fig-2}.
\begin{figure}
\centering
\scalebox{0.5}[0.5]{
  \includegraphics{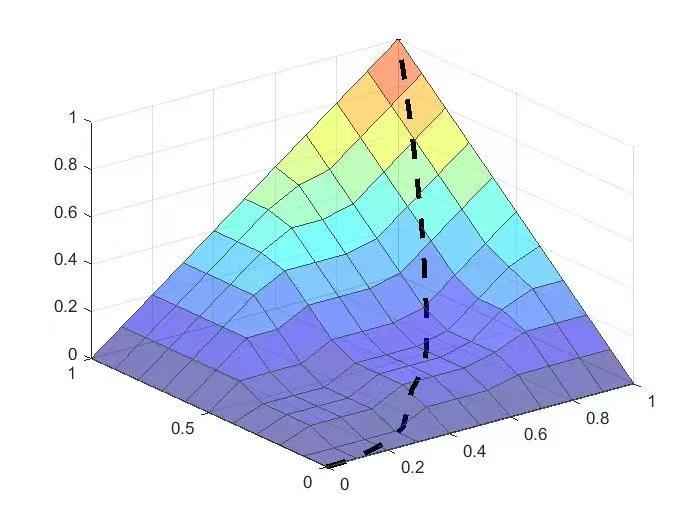}}
  \caption{3D plot of the copula in Example~\ref{Exa-itisacopula}.}
  \label{fig-2}
\end{figure}
\end{exa}

Note that $\Gamma^\prime_\phi(t)\in\{0, 1+\phi^\prime(t)\}$ almost everywhere on the set $\{t\in[0, 1]\mid \Gamma_\phi(t)<\min\{t, \phi(t)\}\}$ in the above two examples. Actually, we have the following result.
\begin{prot}\label{prot-simcri}
Let $\Gamma^\prime_\phi(t)\in\{0, 1+\phi^\prime(t)\}$ almost everywhere on the set $\{t\in[0, 1]\mid \Gamma_\phi(t)<\min\{t, \phi(t)\}\}$. Then $C_2\MySs_\phi C_1$ is a copula.
\end{prot}
\begin{proof}
Let $[x_1, x_2]\subseteq[0, 1]$ be arbitrarily given. If there exists $t\in [x_1, x_2]$ such that $\Gamma_\phi(t)=\min\{t, \phi(t)\}$, then either~\eqref{Condition-1} or~\eqref{Condition-2} holds. If $[x_1, x_2]\subseteq\{t\in[0, 1]\mid \Gamma_\phi(t)<\min\{t, \phi(t)\}\}$, then
\[
\mathbb{V}_{x_1}^{x_2}(\hat{\Gamma}_\phi)\geq\int_{x_1}^{x_2}|\hat{\Gamma}_\phi^\prime(t)|{\rm d}t=\int_{[x_1, x_2]\cap\{t\mid \Gamma^\prime_\phi(t)=0\}}{\rm d}t+ \int_{[x_1, x_2]\cap\{t\mid \Gamma^\prime_\phi(t)= 1+\phi^\prime(t)\}} \phi^\prime(t){\rm d}t
\]
and
\[
\mathbb{V}_{x_1}^{x_2}(\tilde{\Gamma}_\phi)\geq\int_{x_1}^{x_2}|\tilde{\Gamma}_\phi^\prime(t)|{\rm d}t=\int_{[x_1, x_2]\cap\{t\mid \Gamma^\prime_\phi(t)=1+\phi^\prime(t)\}}{\rm d}t+ \int_{[x_1, x_2]\cap\{t\mid \Gamma^\prime_\phi(t)=0\}} \phi^\prime(t){\rm d}t.
\]
Hence, we have
\[
\mathbb{V}_{x_1}^{x_2}(\hat{\Gamma}_\phi)+\mathbb{V}_{x_1}^{x_2}(\tilde{\Gamma}_\phi)\geq\int_{x_1}^{x_2}{\rm d}t+\int_{x_1}^{x_2}\phi^\prime(t){\rm d}t=(x_2-x_1)+(\phi(x_2)-\phi(x_1)),
\]
\emph{i.e.}, ~\eqref{Condition-3} holds. Therefore, at least one of the three inequalities~\eqref{Condition-1}--\eqref{Condition-3} holds, which implies that $C_2\MySs_\phi C_1$ is a copula.
\end{proof}

Whether the converse of Proposition~\ref{prot-simcri} holds is an interesting question for future work.

Now we consider the case that $\phi=\rm{id}$. In this case, both \eqref{Condition-1} and \eqref{Condition-2} reduce to
\begin{equation}\label{Condition1-delta}
\mathbb{V}_{x_1}^{x_2}(\hat{\delta})\geq \hat{\delta}(x_1)+\hat{\delta}(x_2)
\end{equation}
and (\ref{Condition-3}) reduces to
\begin{equation}\label{Condition2-delta}
\mathbb{V}_{x_1}^{x_2}(\hat{\delta})= x_2-x_1.
\end{equation}
If $\hat{\delta}(x_1)=0$, \emph{i.e.}, $\delta(x_1)=x_1$, then \eqref{Condition1-delta} holds. If $\delta(x_1)<x_1$, then there exists $x_2>x_1$ such that $x_2-x_1<\hat{\delta}(x_1)$. As $\hat{\delta}$ is $1$-Lipschitz, we have $\mathbb{V}_{x_1}^{x_2}(\hat{\delta})\leq x_2-x_1<\hat{\delta}(x_1)+\hat{\delta}(x_2)$, \emph{i.e.}, \eqref{Condition1-delta} fails. To ensure $\bar{C}_\delta$ to be a copula, \eqref{Condition2-delta} must hold, while \eqref{Condition2-delta} holds if and only if $|\hat{\delta}^\prime|=1$ almost everywhere on $[x_1, x_2]$. Hence, Theorem~\ref{thrm-CbeCopu} degenerates into Theorem 3.8 of~\cite{KMS24} if $\phi=\rm{id}$.
\begin{corl}
Let $\delta$ be a diagonal function. Then $\bar{C}_\delta$ is a copula if and only if $\delta^\prime(t)\in\{0, 2\}$ almost everywhere on the set $\{x\in[0, 1]\mid \delta(x)<x\}$.
\end{corl}

It is known that $\bar{C}_\delta$ is a copula if and only if $\bar{C}_\delta=K_\delta$ (see~\cite{NQRU08}), where $K_\delta$ is always a copula defined by $K_\delta(x, y)=\min\{x, y, \dfrac{\delta(x)+\delta(y)}{2}\}$. In the general curvilinear case, $K_\delta$ is generalized to $K_{\Gamma_\phi}$ (see~\cite{BMJ19}), where $K_{\Gamma_\phi}$ is defined by
\[
K_{\Gamma_\phi}(x, y)=\min\{x, y, \frac{\Gamma_\phi(x)+\Gamma_\phi(\phi^{-1}(y))}{2}\}.
\]
However, $K_{\Gamma_\phi}$ is not even a quasi-copula. Zou et al.~\cite{ZSX22} proved that $K_{\Gamma_\phi}$ is a copula if and only if both $2t-\Gamma_\phi(t)$ and $2t-\Gamma_\phi(\phi^{-1}(t))$ are increasing. It is easy to see that this condition is equivalent to
\begin{equation}\label{Condition-KisCopu}
|\Gamma_\phi(t_2)-\Gamma_\phi(t_1)|\leq 2\min\{|t_2-t_1|, |\phi(t_2)-\phi(t_1)|\}.
\end{equation}
Let us go back to Example~\ref{Exa-illustrative}, where we gave a curvilinear section $\Gamma_\phi$ such that $C_2\MySs_\phi C_1$ is a copula. However, for such $\Gamma_\phi$, if $t_1, t_2\in [u^*_i, b_i[\,$ for some $i\in I$, then
\[
|\Gamma_\phi(t_2)-\Gamma_\phi(t_1)|=|t_2-t_1|+|\phi(t_2)-\phi(t_1)|>2\min\{|t_2-t_1|, |\phi(t_2)-\phi(t_1)|\}
\]
if $\phi(t_2)-\phi(t_1)\neq t_2-t_1$. Hence, $K_{\Gamma_\phi}$ is not a copula.

Note that~\eqref{Condition-KisCopu} is trivial if $\phi=\rm{id}$. If $\phi\neq \rm{id}$, however, it is quite strict. We cannot even give an example such that $C_2\MySs_\phi C_1$ is a copula and $C_2\MySs_\phi C_1=K_{\Gamma_\phi}$. This motivates the following question:

\emph{Let $\phi\in\Phi$ and $\Gamma_\phi\in\Delta_\phi$. Does it hold that $\phi=\rm{id}$ if $C_2\MySs_\phi C_1$ is a copula and $C_2\MySs_\phi C_1=K_{\Gamma_\phi}$}?

\section{Coincidence of the pointwise supremum of $\mathbb{C}_{\Gamma_\phi}$ and $A_{\Gamma_\phi}$}\label{Sec-ChaCoin}

In this section, we characterize when $C_2\MySs_\phi C_1=A_{\Gamma_\phi}$, as the following theorem shows.
\begin{thrm}\label{thrm-C=A}
Let $\phi\in\Phi$ and $\Gamma_\phi\in\Delta_\phi$. Then $C_2\MySs_\phi C_1=A_{\Gamma_\phi}$ if and only if, for any
$(x, y)\in[0, 1]^2$, the following two conditions are satisfied:
\begin{itemize}
\item[\rm (i)] $y-x>\max\{\tilde{\Gamma}_\phi(x), \tilde{\Gamma}_\phi(\phi^{-1}(y))\}$ if $y>\phi(x)$ and $\min\limits_{t\in [x, \phi^{-1}(y)]}\tilde{\Gamma}_\phi(t)<\min\{\tilde{\Gamma}_\phi(x), \tilde{\Gamma}_\phi(\phi^{-1}(y))\}$;

\item[\rm (ii)] $x-y>\max\{\hat{\Gamma}_\phi(\phi^{-1}(y)), \hat{\Gamma}_\phi(x)\}$ if $y<\phi(x)$ and $\min\limits_{t\in [\phi^{-1}(y), x]}\hat{\Gamma}_\phi(t)<\min\{\hat{\Gamma}_\phi(x), \hat{\Gamma}_\phi(\phi^{-1}(y))\}$.
\end{itemize}
\end{thrm}
\begin{proof}
\emph{Necessity:} (i) Suppose that $y>\phi(x)$ and $\min\limits_{t\in [x, \phi^{-1}(y)]}\tilde{\Gamma}_\phi(t)<\min\{\tilde{\Gamma}_\phi(x), \tilde{\Gamma}_\phi(\phi^{-1}(y))\}$. Clearly, there exist
$t_0\in\,]x, \phi^{-1}(y)[\,$ and $t^*\in [x, \phi^{-1}(y)]$ such that
\[\tilde{\Gamma}_\phi(t_0)=\min\limits_{t\in[x, \phi^{-1}(y)]}\tilde{\Gamma}_\phi(t), \tilde{\Gamma}_\phi(t^*)=\max\limits_{t\in[x, \phi^{-1}(y)]}\tilde{\Gamma}_\phi(t)\]
and $t_0\neq t^*$. We consider two cases:
\begin{itemize}
\item[\rm (a)] $t_0>t^*$.

 In this case, we have $A_{\Gamma_\phi}(t^*, y)=\min\{t^*, y-\tilde{\Gamma}_\phi(t^*)\}$. Since
 \[
 \mathbb{V}_{t^*}^{\phi^{-1}(y)}(\tilde{\Gamma}_\phi)\geq\tilde{\Gamma}_\phi(\phi^{-1}(y))+\tilde{\Gamma}_\phi(t^*)-2\tilde{\Gamma}_\phi(t_0)
 >\tilde{\Gamma}_\phi(t^*)-\tilde{\Gamma}_\phi(\phi^{-1}(y)),
 \]
 it follows that $f_2(t^*, y)<y-\tilde{\Gamma}_\phi(t^*)$. To ensure
 \[(C_2\MySs_\phi C_1)(t^*, y)=\min\{t^*, f_2(t^*, y)\}=A_{\Gamma_\phi}(t^*, y),\]
 it must hold that $t^*<y-\tilde{\Gamma}_\phi(t^*)$. Hence, we have
 \[y-x\geq y-t^*>\tilde{\Gamma}_\phi(t^*)\geq\max\{\tilde{\Gamma}_\phi(x), \tilde{\Gamma}_\phi(\phi^{-1}(y))\}.\]

 \item[\rm (b)] $t_0<t^*$.

 In this case, we have $A_{\Gamma_\phi}(x, \phi(t^*))=\min\{x, \phi(t^*)-\tilde{\Gamma}_\phi(t^*)\}$. Since
 \[
 \mathbb{V}_x^{t^*}(\tilde{\Gamma}_\phi)\geq\tilde{\Gamma}_\phi(t^*)+\tilde{\Gamma}_\phi(x)-2\tilde{\Gamma}_\phi(t_0)
 >\tilde{\Gamma}_\phi(t^*)-\tilde{\Gamma}_\phi(x),
 \]
 it follows that $f_2(x, \phi(t^*))<\phi(t^*)-\tilde{\Gamma}_\phi(t^*)$. To ensure
 \[(C_2\MySs_\phi C_1)(x, \phi(t^*))=\min\{x, f_2(x, \phi(t^*))\}=A_{\Gamma_\phi}(x, \phi(t^*)),\]
  it must hold that $x<\phi(t^*)-\tilde{\Gamma}_\phi(t^*)$. Hence, we have
  \[y-x\geq \phi(t^*)-x>\tilde{\Gamma}_\phi(t^*)\geq\max\{\tilde{\Gamma}_\phi(x), \tilde{\Gamma}_\phi(\phi^{-1}(y))\}.\]
\end{itemize}

(ii) The proof is similar to that of (i).

\emph{Sufficiency:} Suppose that $(x, y)\in [0, 1]^2$. If $y=\phi(x)$, then
\[(C_2\MySs_\phi C_1)(x, y)=A_{\Gamma_\phi}(x, y)=\Gamma_\phi(x).\]
Now we consider the case that $y<\phi(x)$ (the proof is similar if $y>\phi(x)$). We distinguish two possible cases to prove that $(C_2\MySs_\phi C_1)(x, y)=A_{\Gamma_\phi}(x, y)$:
\begin{itemize}
\item[\rm (a)] For any $[y_1, x_1]\subseteq [\phi^{-1}(y), x]$, we have $\min\{\hat{\Gamma}_\phi(y_1), \hat{\Gamma}_\phi(x_1)\}=\min\limits_{t\in[y_1, x_1]}\hat{\Gamma}_\phi(t)$.

    In this case, there exists $t_0\in [\phi^{-1}(y), x]$ such that $\hat{\Gamma}_\phi(t)$ is increasing on $[\phi^{-1}(y), t_0]$ and decreasing on $[t_0, x]$. Hence,
    \[
    \mathbb{V}_{\phi^{-1}(y)}^x\hat{\Gamma}_\phi(t)=2\hat{\Gamma}_\phi(t_0)-\hat{\Gamma}_\phi(\phi^{-1}(y))-\hat{\Gamma}_\phi(x),
    \]
    which implies that $f_1(x, y)=x-\hat{\Gamma}_\phi(t_0)$ and
    \[
    (C_2\MySs_\phi C_1)(x, y)=\min\{y, x-\hat{\Gamma}_\phi(t_0)\}=\min\{y, x-\max\limits_{t\in[\phi^{-1}(y), x]}\hat{\Gamma}_\phi(t)\}=A_{\Gamma_\phi}(x, y).
    \]

\item[\rm (b)] There exists $[y_1, x_1]\subseteq [\phi^{-1}(y), x]$ with $\min\{\hat{\Gamma}_\phi(y_1), \hat{\Gamma}_\phi(x_1)\}>\min\limits_{t\in[y_1, x_1]}\hat{\Gamma}_\phi(t)$.

    In this case, there exists $t_1\in\,]y_1, x_1[\,$ such that $\hat{\Gamma}_\phi(t_1)=\min\limits_{t\in[y_1, x_1]}\hat{\Gamma}_\phi(t)$. Stipulate
    \[y_0=\inf\{s\in\,]y_1, t_1]\mid \hat{\Gamma}_\phi(t)=\hat{\Gamma}_\phi(t_1), \forall t\in[s, t_1]\}\]
    and
    \[x_0=\sup\{s\in[t_1, x_1[\,\mid \hat{\Gamma}_\phi(t)=\hat{\Gamma}_\phi(t_1), \forall t\in[t_1, s]\}.\]
    It is easy to see that $[y_0, x_0]\subseteq\,]y_1, x_1[\,$ and $\hat{\Gamma}_\phi(t)=\hat{\Gamma}_\phi(t_1)$ for all $t\in [y_0, x_0]$. For sufficiently small $\epsilon>0$, there exist $y_\epsilon\in\,]y_0-\epsilon, y_0[\,$ and $x_\epsilon\in\,]x_0, x_0+\epsilon[\,$ such that $\hat{\Gamma}_\phi(y_\epsilon)>\hat{\Gamma}_\phi(t_1)$ and $\hat{\Gamma}_\phi(x_\epsilon)>\hat{\Gamma}_\phi(t_1)$, \emph{i.e.},
    \[
    \min_{t\in[y_\epsilon, x_\epsilon]}\hat{\Gamma}_\phi(t)=\hat{\Gamma}_\phi(t_1)
    <\min\left\{\hat{\Gamma}_\phi(y_\epsilon), \hat{\Gamma}_\phi(x_\epsilon)\right\}.
    \]
    It follows from condition (ii) that
    \[
    x_\epsilon-\phi(y_\epsilon)>\max\{\hat{\Gamma}_\phi(y_\epsilon), \hat{\Gamma}_\phi(x_\epsilon)\}>\hat{\Gamma}_\phi(t_1).
    \]
    Letting $\epsilon\to 0$, we have $x_0-\phi(y_0)\geq\hat{\Gamma}_\phi(t_1)$.
    Hence,
    \begin{eqnarray*}
    (C_2\MySs_\phi C_1)(x_0, \phi(y_0))&=&\min\left\{\phi(y_0), x_0-\frac{1}{2}\Big(\mathbb{V}_{y_0}^{ x_0}(\hat{\Gamma}_\phi)+\hat{\Gamma}_\phi(x_0)+\hat{\Gamma}_\phi(y_0)\Big)\right\}\\
    &=&\min\{\phi(y_0), x_0-\hat{\Gamma}_\phi(t_1)\}=\phi(y_0).
    \end{eqnarray*}
  As $C_2\MySs_\phi C_1$ is a quasi-copula, it is increasing and $1$-Lipschitz, and thus
   \[
   (C_2\MySs_\phi C_1)(x, y)\geq(C_2\MySs_\phi C_1)(x_0, y)\geq(C_2\MySs_\phi C_1)(x_0, \phi(y_0))-(\phi(y_0)-y)=y.
   \]
  Therefore, $(C_2\MySs_\phi C_1)(x, y)=y=A_{\Gamma_\phi}(x, y)$.
\end{itemize}
We conclude that $C_2\MySs_\phi C_1=A_{\Gamma_\phi}$.
\end{proof}

Conditions (i) and (ii) in Theorem~\ref{thrm-C=A} can be rephrased as follows:
for any interval $[x, y]\subseteq[0, 1]$, the following two conditions are satisfied:
\begin{itemize}
\item[\rm (i)] $\phi(y)-x>\max\{\tilde{\Gamma}_\phi(x), \tilde{\Gamma}_\phi(y)\}$ if $\min\limits_{t\in [x, y]}\tilde{\Gamma}_\phi(t)<\min\{\tilde{\Gamma}_\phi(x), \tilde{\Gamma}_\phi(y)\}$;

\item[\rm (ii)] $y-\phi(x)>\max\{\hat{\Gamma}_\phi(x), \hat{\Gamma}_\phi(y)\}$ if $\min\limits_{t\in [x, y]}\hat{\Gamma}_\phi(t)<\min\{\hat{\Gamma}_\phi(x), \hat{\Gamma}_\phi(y)\}$.
\end{itemize}

We illustrate these two conditions as follows.
\begin{exa}
(i) Consider $\hat{\Gamma}_\phi$ in Example~\ref{Exa-itisacopula}. Since $\dfrac{2}{5}-(\dfrac{2}{5})^2<\dfrac{2}{5}-\dfrac{26}{225}$, there exist $x_0\in\,]\dfrac{1}{3}, \dfrac{2}{5}[\,$ and $y_0\in\, ]\dfrac{2}{5}, t_0[\,$ with $y_0-\phi(x_0)<\hat{\Gamma}_\phi(\dfrac{2}{5})$. Note that $\min\limits_{t\in [x_0, y_0]}\hat{\Gamma}_\phi(t)=\hat{\Gamma}_\phi(\dfrac{2}{5})<\min\{\hat{\Gamma}_\phi(x_0), \hat{\Gamma}_\phi(y_0)\}$, but $y_0-\phi(x_0)<\max\{\hat{\Gamma}_\phi(x_0), \hat{\Gamma}_\phi(y_0)\}$, \emph{i.e.}, condition (ii) of Theorem~\ref{thrm-C=A} fails. Hence, $C_2\MySs_\phi C_1\neq A_{\Gamma_\phi}$.

(ii) Consider $\phi\in\Phi$ and $\Gamma_\phi\in\Delta_\phi$ defined by
\[
\phi(t)=\left \{
        \begin {array}{ll}
         \dfrac{\sqrt{t}}{2}
                  &\quad \text{\rm if\ \ $t\in\, [0, \dfrac{1}{4}]$} \\[2mm]
         t
                  &\quad \text{\rm if\ \ $t\in\, ]\dfrac{1}{4}, \dfrac{3}{4}]$} \\[2mm]
         \dfrac{4t^2+3}{7}
                  &\quad \text{\rm if\ \ $t\in\,]\dfrac{3}{4}, 1]\,$}
        \end {array}
       \right.
\]
and
\[
\Gamma_\phi(t) =\left \{
        \begin {array}{ll}
         t-t^2
                  &\quad \text{\rm if\ \ $t\in [0, \dfrac{1}{4}]$} \\[2mm]
           \dfrac{9t}{8}-\dfrac{3}{32}
                  &\quad \text{\rm if\ \ $t\in\,]\dfrac{1}{4}, \dfrac{1}{2}]\,$} \\[2mm]
           t-\dfrac{1}{32}
                  &\quad \text{\rm if\ \ $t\in\,]\dfrac{1}{2}, \dfrac{3}{4}]\,$} \\[2mm]
           \dfrac{t}{4}+\dfrac{17}{32}
                  &\quad \text{\rm if\ \ $t\in\,]\dfrac{3}{4}, \dfrac{7}{8}]\,$}\\[2mm]
           2t-1
                  &\quad \text{\rm if\ \ $t\in\,]\dfrac{7}{8}, 1].\,$}
        \end {array}
       \right.
\]
Since
\[
\tilde{\Gamma}_\phi(t) =\left \{
        \begin {array}{ll}
        \dfrac{\sqrt{t}}{2}-t+t^2
                  &\quad \text{\rm if\ \ $t\in [0, \dfrac{1}{4}]$} \\[2mm]
        -\dfrac{t}{8}+\dfrac{3}{32}
                  &\quad \text{\rm if\ \ $t\in\,]\dfrac{1}{4}, \dfrac{1}{2}]\,$} \\[2mm]
           \dfrac{1}{32}
                  &\quad \text{\rm if\ \ $t\in\,]\dfrac{1}{2}, \dfrac{3}{4}]\,$} \\[2mm]
           \dfrac{4t^2}{7}-\dfrac{t}{4}-\dfrac{23}{224}
                  &\quad \text{\rm if\ \ $t\in\,]\dfrac{3}{4}, \dfrac{7}{8}]\,$}\\[2mm]
           \dfrac{4t^2}{7}-2t+\dfrac{10}{7}
                  &\quad \text{\rm if\ \ $t\in\,]\dfrac{7}{8}, 1],\,$}
        \end {array}
       \right.
\]
for any interval $[x, y]\subseteq[0, 1]$ with $\min\limits_{t\in [x, y]}\tilde{\Gamma}_\phi(t)<\min\{\tilde{\Gamma}_\phi(x), \tilde{\Gamma}_\phi(y)\}$, we have $x<\dfrac{1}{2}$ and $y>\dfrac{3}{4}$. Hence,
\[\phi(y)-x>\dfrac{1}{4}>\max\{\tilde{\Gamma}_\phi(\dfrac{3-\sqrt{5}}{8}), \tilde{\Gamma}_\phi(\dfrac{7}{8})\}\geq\max\{\tilde{\Gamma}_\phi(x), \tilde{\Gamma}_\phi(y)\},\]
\emph{i.e.}, condition (i) of Theorem~\ref{thrm-C=A} holds. Similarly, condition (ii) of Theorem~\ref{thrm-C=A} also holds.
So $C_2\MySs_\phi C_1=A_{\Gamma_\phi}$.
\end{exa}

Recall that a curvilinear section $\Gamma_\phi$ is said to be \emph{$\phi$-simple} if
\[\hat{\Gamma}_\phi(\lambda x+(1-\lambda) y)\geq\min\{\hat{\Gamma}_\phi(x), \hat{\Gamma}_\phi(y)\}\]
and
\[\tilde{\Gamma}_\phi(\lambda x+(1-\lambda) y)\geq\min\{\tilde{\Gamma}_\phi(x), \tilde{\Gamma}_\phi(y)\}\]
hold for any $x, y, \lambda\in [0, 1]$ (see~\cite{XWZJ23}). Clearly, if $\Gamma_\phi$ is $\phi$-simple, then, for any
$(x_1, x_2)\in[0, 1]^2$ with $x_1<x_2$, we have

\centerline{$\min\{\hat{\Gamma}_\phi(x_1), \hat{\Gamma}_\phi(x_2)\}=\min\limits_{t\in[x_1, x_2]}\hat{\Gamma}_\phi(t)$ and $\min\{\tilde{\Gamma}_\phi(x_1), \tilde{\Gamma}_\phi(x_2)\}=\min\limits_{t\in[x_1, x_2]}\tilde{\Gamma}_\phi(t)$.}

\noindent Hence, conditions (i) and (ii) in Theorem~\ref{thrm-C=A} are naturally satisfied, and thus
$C_2\MySs_\phi C_1=A_{\Gamma_\phi}$.
\begin{corl}\cite{XWZJ23}
Let $\phi\in\Phi$ and $\Gamma_\phi\in\Delta_\phi$. If $\Gamma_\phi$ is $\phi$-simple, then $C_2\MySs_\phi C_1=A_{\Gamma_\phi}$.
\end{corl}


\end{document}